\documentclass{amsart}
\usepackage{latexsym}
\usepackage{amstext}
\usepackage{amssymb}
\usepackage{amsfonts}
\usepackage{amsthm}
\usepackage{amsopn}
\usepackage{amsbsy}
\usepackage{layout}
\usepackage{color}
\usepackage{amsmath}
\usepackage{graphicx}
\usepackage{bm}
\usepackage{yfonts}

\newtheorem{deff}{Definition}[section]

\newtheorem{theorem}[deff]{Theorem}
\newtheorem{corollary}[deff]{Corollary}

\newtheorem{proposition}[deff]{Proposition}

\newtheorem{Claim}[deff]{Claim}
\newtheorem{em-example}[deff]{Example}
\newtheorem{em-def}[deff]{Definition}        
\newtheorem{em-remark}[deff]{Remark}         
\newtheorem{em-question}[deff]{Question}

\newtheorem{problem}[deff]{Problem}

\newenvironment{example}{\begin{em-example} \em }{ \end{em-example}}

\newenvironment{remark}{\begin{em-remark} \em }{\end{em-remark}}

\newcommand{\acal}{{\mathcal A}}

\newcommand{\dcal}{\Delta}

\DeclareMathSymbol{\res}{\mathord}{AMSa}{"16}

\def\:{\nobreak \hskip .1111em\mathpunct {}\nonscript \mkern
   -\thinmuskip {:}\hskip .3333emplus.0555em\relax}

\catcode`\@=12

\def\N{{\mathbb N}}
\def\R{{\mathbb R}}

\def\P{{\mathbb P}}

\def\cont{\mathfrak c}

\begin{document}
\title[A characterization of $X \in \dcal$]{A characterization of $X$ for which spaces $C_p(X)$ are distinguished and its applications}
\date{\today}

\author{Jerzy K\c akol}
\address{ Faculty of Mathematics and Informatics, A. Mickiewicz University,
61-614 Pozna\'{n}, Poland and Institute of Mathematics Czech Academy of Sciences, Prague, Czech \newline Republic}
\email{kakol@amu.edu.pl}
\author{Arkady Leiderman}
\address{Department of Mathematics, Ben-Gurion University of the Negev, Beer Sheva,\newline Israel}
\email{arkady@math.bgu.ac.il}

\begin{abstract}
We prove that
the locally convex space $C_{p}(X)$ of continuous real-valued functions on a Tychonoff space $X$
 equipped with the topology of pointwise convergence  is distinguished
if and only if $X$ is a $\dcal$-space in the sense of \cite {Knight}.
 As an application of this characterization theorem we obtain the following results:

1) If $X$ is a \v{C}ech-complete (in particular, compact) space such that $C_p(X)$ is distinguished, then $X$ is scattered.
2) For every separable compact space of the Isbell--Mr\'owka type $X$, the space $C_p(X)$ is distinguished.
3) If $X$ is the compact space of ordinals $[0,\omega_1]$, then $C_p(X)$ is not distinguished.

We observe that the existence of an uncountable separable metrizable space $X$ such that $C_p(X)$ is distinguished,
 is independent of  ZFC.
We explore  also the question to which extent the class of $\dcal$-spaces is invariant under
basic topological operations.
\end{abstract}

\thanks{The research for the first named author is supported  by the GA\v{C}R project 20-22230L and RVO: 67985840.
He thanks  also the Center For Advanced Studies in Mathematics of Ben-Gurion University of the Negev for financial support during his visit in 2019.}

\maketitle

\section{Introduction}\label{intro}

Following  J. Dieudonn\'{e} and
L. Schwartz  \cite{dieudonne} a locally convex space (lcs) $E$ is called \emph{distinguished} if every bounded subset of the bidual of $E$
 in the weak$^{*}$-topology is contained in the closure of the weak$^{*}$-topology  of  some bounded subset of $E$. 
Equivalently, a lcs $E$ is distinguished if and only if the strong dual of $E$  (i.e. the topological  dual of $E$ endowed with the strong topology) 
is \emph{barrelled}, (see \cite[8.7.1]{Ko}).  A. Grothendieck  \cite{grothendieck} proved that a metrizable lcs $E$ is distinguished 
if and only if its strong dual is \emph{bornological}. 
 We refer the reader to survey articles \cite{BB1} and \cite{BB2} which present several more modern results about distinguished metrizable and Fr\'echet lcs.

Throughout the article, all topological spaces are assumed to be Tychonoff and infinite.
By $C_{p}(X)$ and $C_{k}(X)$  we mean the spaces of all real-valued continuous functions on a Tychonoff space $X$
 endowed with the topology of pointwise convergence and the compact-open topology, respectively.
By a \emph{bounded set} in a topological vector space (in particular, $C_{p}(X)$) we understand any set which is absorbed by every $0$-neighbourhood.

For spaces $C_{p}(X)$ we proved in \cite{FKLS} the following theorem (the equivalence $(1) \Leftrightarrow (4)$ has been obtained in \cite{fe-ka}).
\begin{theorem}\label{Theor:characterization}
For a Tychonoff space $X$, the following conditions are equivalent{\rm:}
\begin{enumerate}
\item[{\rm (1)}] $C_p(X)$ is distinguished.
\item[{\rm (2)}]  $C_{p}\left( X\right) $ is a large subspace of
$\mathbb{R}^{X}$, i.e. for every bounded set $A$ in $\mathbb{R}^{X}$ there exists a bounded set $B$ in $C_{p}(X)$ such that $A\subset cl_{\R^X}(B)$.
\item [{\rm (3)}] For every  $f \in \R^X$ there is a bounded set $B \subset C_p(X)$ such that
$f \in cl_{\R^X}(B)$.
\item[{\rm (4)}] The strong dual of the space $C_{p}(X)$ carries the finest locally convex topology.
\end{enumerate}
\end{theorem}
Several examples of $C_{p}(X)$ with(out) distinguished property have been provided in papers \cite{fe-ka}, \cite{fe-ka-sa} and \cite{FKLS}. 
The aim of this research is to continue our initial work on distinguished spaces $C_{p}(X)$.

The following concept plays a key role in our paper.  We show its applicability for the studying of distinguished spaces $C_{p}(X)$.
\begin{deff}{\rm (\cite{Knight})}\label{def:Delta}
A topological space $X$ is said to be a $\Delta$-space if for every decreasing sequence $\{D_n: n \in \omega\}$
of subsets of $X$ with empty intersection, there is a decreasing sequence $\{V_n: n \in \omega\}$ consisting of open
subsets of $X$, also with empty intersection, and such that $D_n \subset V_n$ for every $n \in \omega$.
\end{deff}

We should mention that R. W. Knight \cite{Knight} called all topological spaces $X$ satisfying the above Definition \ref{def:Delta} by $\Delta$-sets.
The original definition of a $\Delta$-set of the real line $\R$ is due to G. M. Reed and E. K. van Douwen (see \cite{Reed}).
In this paper, for general topological spaces satisfying Definition \ref{def:Delta} we reserve the term \emph{$\Delta$-space}. 
The class of all $\Delta$-spaces is denoted by $\Delta$.

In Section \ref{description} we give an intrinsic description of all spaces $X \in \Delta$.
One of the main results of our paper, Theorem \ref{Theor:description} says that $X$ is a $\Delta$-space
 if and only if $C_{p}(X)$ is a distinguished space. 
This characterization theorem has been applied systematically for obtaining a range of results from our paper.

Our main result in Section \ref{compact} states that a \v{C}ech-complete (in particular, compact) $X \in \dcal$ must be scattered.
  A very natural question arises what are those scattered compact spaces $X \in \dcal$.
In view of Theorem \ref{Theor:description}, it is known that a Corson compact $X$ belongs to the class $\dcal$ if and only if $X$ is a scattered Eberlein compact space \cite{FKLS}.
With the help of Theorems \ref{Theor:description} and \ref{Theor:union}
 we show that the class $\Delta$ contains also all separable compact spaces of the Isbell--Mr\'owka type.
Nevertheless, as we demonstrate in  Section \ref{compact}, there are compact scattered spaces $X \notin \Delta$  (for example, the compact space $[0,\omega_{1}]$).

Section \ref{metr} deals with the questions about metrizable spaces $X\in \Delta$.
We notice that every $\sigma$-scattered metrizable space $X$ belongs to the class $\Delta$.
For separable metrizable  spaces $X$, our analysis reveals a tight connection between distinguished $C_p(X)$ and well-known set-theoretic problems about special subsets of the real line $\R$.
We observe that the existence of an uncountable separable metrizable space $X\in\Delta$ is independent of ZFC and it is equivalent to the existence of
a separable countably paracompact nonnormal Moore space. 
We refer readers to \cite{Nyikos} for the history of the normal Moore problem.

In Section \ref{problems} we study whether the class $\Delta$ is invariant under the basic topological operations:
subspaces, (quotient) continuous images, finite/countable unions and finite products. We pose several new open problems.

\section{Description Theorem}\label{description}

In this section we provide an intrinsic description of $X \in \dcal$. For the reader's convenience we recall some relevant terminology.
\begin{enumerate}
\item[\rm (a)] A disjoint cover $\{X_{\gamma}: \gamma \in \Gamma\}$ of $X$ is called a {\it partition} of $X$.
\item[\rm (b)] A collection of sets $\{U_{\gamma}: \gamma \in \Gamma\}$ is called an {\it expansion} of 
a collection of sets $\{X_{\gamma}: \gamma \in \Gamma\}$ in $X$
if $X_{\gamma} \subseteq U_{\gamma} \subseteq X$ for every index $\gamma \in \Gamma$.
\item[\rm (c)] A collection of sets $\{U_{\gamma}: \gamma \in \Gamma\}$ is called {\it point-finite} if no point belongs to infinitely many $U_{\gamma}$ -s.
\end{enumerate}

\begin{theorem}\label{Theor:description}
For a Tychonoff space $X$, the following conditions are equivalent{\rm:}
\begin{enumerate}
\item[{\rm (1)}] $C_p(X)$ is distinguished.
\item[{\rm (2)}] Any countable partition of $X$ admits a point-finite open expansion in $X$.
\item[{\rm (3)}] Any countable disjoint collection of subsets of $X$ admits a point-finite open expansion in $X$.
\item[{\rm (4)}] $X$ is a $\Delta$-space.
\end{enumerate}
\end{theorem}
\begin{proof} Observe that every collection of pairwise disjoint subsets of $X$, $\{X_{\gamma}: \gamma \in \Gamma\}$ can be extended to a partition
by adding a single set $X_{\ast} = X \setminus \bigcup \{X_{\gamma}: \gamma \in \Gamma\}$. If the obtained partition admits a point-finite open expansion in $X$,
then removing one open set we get a point-finite open expansion of the original disjoint collection. This shows evidently the equivalence (2) $\Leftrightarrow$ (3).

Assume now that (3) holds. Let $\{D_n: n \in \omega\}$ be a decreasing sequence subsets of $X$ with empty intersection.
 Define $X_n = D_n \setminus D_{n+1}$ for each $n \in \omega$. By assumption, a disjoint collection $\{X_n: n \in \omega\}$
 admits a point-finite open expansion $\{U_n: n \in \omega\}$ in $X$. Then $\{V_n = \bigcup \{U_i: i \geq n\}: n \in \omega\}$
is an open decreasing expansion in $X$  with empty intersection. This proves the implication (3) $\Rightarrow$ (4).

Next we show (4) $\Rightarrow$ (2).
Let $\{X_n: n \in \omega\}$ be any countable partition of $X$. Define $D_0 = X$ and $D_n = X \setminus \bigcup\{X_i: i < n\}$.
Then $X_n \subset D_n$ for every $n$, the sequence $\{D_n: n \in \omega\}$ is decreasing and its intersection is empty.
 Assuming (4), we find an open decreasing expansion  $\{U_n: n \in \omega\}$ of $\{D_n: n \in \omega\}$ in $X$ such that $\bigcap\{U_n: n \in \omega\} = \emptyset$.
For every $x \in X$ there is $n$ such that $x \notin U_m$ for each $m > n$, it means that $\{U_n: n \in \omega\}$ is a point-finite expansion of
$\{X_n: n \in \omega\}$ in $X$. This finishes the proof (3) $\Rightarrow$ (4) $\Rightarrow$ (2) $\Leftrightarrow$ (3).

Now we prove the implication (1) $\Rightarrow$ (2). Let $\{X_n: n \in \omega\}$ be any countable partition of $X$.
Fix any function $f \in \R^X$ which satisfies the following conditions: for each $n \in \omega$ and every $x \in X_n$ the value of $f(x)$ is greater than $n$.
By assumption, there is a bounded subset $B$ of $C_p(X)$ such that $f \in cl_{\R^X}(B)$.
Hence, for every $n \in \omega$ and every point $x \in X_n$,  there exists $f_x \in B$ such that
$f_x(x) > n$. But $f_x$ is a continuous function, therefore there is an open neighbourhood $U_x \subset X$ of $x$ such that $f_x(y) > n$ for every $y \in U_x$.
We define an open set $U_n \subset X$ as follows: $U_n = \bigcup\{U_x: x \in X_n\}$. Evidently, $X_n \subseteq U_n$ for each $n \in \omega$.
If we assume that the open expansion $\{U_n : n \in \omega\}$ is not point-finite, then there exists a point $y \in X$ such that
 there are infinitely many numbers $n$ with $y \in U_{x_n}$ for some $x_n \in X_n$. This means that $\sup\{g(y): g \in B \} = \infty$,
which contradicts the boundedness of $B$.

It remains to prove (2) $\Rightarrow$ (1).
By Theorem \ref{Theor:characterization}, we need to show that for every mapping $f \in \R^X$
 there is a bounded set $B \subset C_p(X)$ such that $f \in cl_{\R^X}(B)$.
If  there exists a constant $r > 0$ such that $\sup\{|f(x)|: x \in X\} < r$,
 then we take $B=\{h \in C(X): \sup\{|h(x)|: x \in X\} < r \}$. It is easy to see that $B$ is as required.

Let $f \in \R^X$ be unbounded. Denote by $Y_0=\emptyset$ and $Y_n =\{x \in X: n-1 \leq |f(x)| < n\}$ for each non-zero $n \in \omega$.
Define $\varphi: X \rightarrow \omega$ by the rule: if $Y_n \neq \emptyset$ then $\varphi(x) = n$ for every $x\in Y_n$.
So, $|f| < \varphi$. Put $X_{n}=\varphi^{-1}(n)$ for each $n \in \omega$. Note that some sets $X_{n}$ might happen to be empty,
but the collection $\{X_{n}: n \in \omega\}$ is a partition of $X$ with countably many nonempty $X_n$ -s.
By our assumption, there exists a point-finite open expansion $\{U_{n}:n\in\omega\}$ of the partition $\{X_{n}:n\in\omega\}$. 
 Define $F:X\rightarrow\omega$ by $F(x)=\max\{n:x\in U_{n}\}$. Obviously, $f < F$. Finally, we define $B = \{h\in C_{p}(X): |h|\leq F\}$.
 Then $f \in cl_{\R^X}(B)$, because for every finite subset $K \subset X$ there is a function $h \in B$
such that $f\restriction_{K} = h\restriction_{K}$.
 Indeed, given a finite subset $K \subset X$, let $\{V_{x}:x\in K\}$ be the family of pairwise disjoint open sets
 such that $x\in V_{x}\subset U_{\varphi(x)}$ for every $x\in K$.  For each $x\in K$, fix a continuous function
 $h_{x}:X\rightarrow [-\varphi(x),\varphi(x)]$ such that $h_{x}(x)= f(x)$ and $h_{x}$ is equal to the constant value $0$ on the closed set $X \setminus V_{x}$.
 One can verify that $h=\Sigma_{x\in K}h_{x} \in B$ is as required.
\end{proof}

Below we present a straightforward application of Theorem \ref{Theor:description}.

\begin{corollary}{\rm (\cite{FKLS})}\label{cor:subspace}
Let $Z$ be any subspace of $X$. If $X$ belongs to the class $\Delta$, then $Z$ also belongs to the class $\Delta$.
\end{corollary}
\begin{proof}
If $\{Z_{\gamma}: \gamma \in \Gamma\}$ is any collection of pairwise disjoint subsets of $Z$
and there exists a point-finite open expansion $\{U_{\gamma}: \gamma \in \Gamma\}$ in $X$,
then obviously $\{U_{\gamma} \cap Z: \gamma \in \Gamma\}$ is a point-finite expansion consisting of the sets relatively open in $Z$.
It remains to apply Theorem \ref{Theor:description}.
\end{proof}

The last result can be reversed, assuming that $X \setminus Z$ is finite.
\begin{proposition}\label{prop:subspace}
Let $Z$ be a subspace of $X$  such that $Y = X \setminus Z$ is finite.
If $Z$ belongs to the class $\Delta$, then  $X$ belongs to $\Delta$ as well.
\end{proposition}
\begin{proof}
Let $\{X_n: n \in \omega\}$ be any countable collection of pairwise disjoint subsets of $X$.
Denote by $F$ the set of those $n \in \omega$ such that $X_n \bigcap Y  \neq \emptyset$.
There might be only finitely many $X_n$ -s which intersect the finite set $Y$, hence  $F \subset \omega$ is finite.
If $n \in F$, then we simply declare that $U_n$ is equal to $X$.
Consider the subcollection $\{X_n: n \in \omega \setminus F\}$. It is a countable collection of pairwise disjoint subsets of $Z$.
Since $Z \in \dcal$, by Theorem \ref{Theor:description},
there is a point-finite open expansion $\{U_n: n \in \omega \setminus F\}$ in $Z$.
 Observe that $Z$ is open in $X$, therefore all those $U_n$ -s remain open in $X$.
Bringing all $U_n$ -s of both sorts together we obtain
a point-finite open expansion $\{U_n: n \in \omega \}$ in $X$. Finally, $X \in \dcal$, by Theorem \ref{Theor:description}.
\end{proof}

\begin{remark} \label{Rem:scant}
The following applicable concept  has been re-introduced  in \cite{FKLS}.
A family $\left\{\mathcal{N}_{x}:x\in X\right\}$ of
subsets of a Tychonoff space $X$ is called a \emph{scant cover} for $X$ if each $\mathcal{N}_{x}$ is an open neighbourhood of $x$ and for each $u\in X$
 the set $X_{u} =\left\{ x\in X:u\in \mathcal{N}_{x}\right\}$ is finite.
 \footnote{ The referee kindly informed the authors  that
 this notion also is known in the literature  under the name \emph{the point-finite neighbourhood assignment}.}

 Our Theorem \ref{Theor:description} generalizes one of the  results obtained in \cite{FKLS} stating that if $X$
admits a scant cover $\left\{ \mathcal{N}_{x}:x\in X\right\}$ then $C_{p}\left(X\right)$ is distinguished.
Indeed, let $\{X_{\gamma}: \gamma \in \Gamma\}$ be any collection of pairwise disjoint subsets of $X$.
Define $U_{\gamma} = \bigcup \{\mathcal{N}_{x}: x\in X_{\gamma}\}$.
It is easily seen that $\{U_{\gamma}: \gamma \in \Gamma\}$ is a point-finite open expansion in $X$, by definition of a scant cover.
Applying Theorem \ref{Theor:description}, we conclude that $C_{p}\left( X\right) $\ is distinguished.
\end{remark}

\section{Applications to compact spaces $X \in \dcal$}\label{compact}

First we  recall a few definitions and facts (probably well-known) which will be used in the sequel.
 A space $X$ is said to be {\it scattered} if every nonempty
subset $A$ of $X$ has an isolated point in $A$. 
Denote by $A^{(1)}$  the set of all non-isolated (in $A$) points of $A \subset X$. 
For ordinal numbers $\alpha$, the $\alpha$-th derivative of a topological space $X$
is defined by transfinite induction as follows. 

$X^{(0)} = X$;\,
$X^{(\alpha+1)} = (X^{(\alpha)})^{(1)}$;\,
$X^{(\gamma)} = \bigcap_{\alpha<\gamma} X^{(\alpha)}$ for limit ordinals $\gamma$.

For a scattered space $X$, the smallest ordinal $\alpha$ such that $X^{(\alpha)}=\emptyset$ is called the \emph{scattered height} of $X$ and is denoted by
$ht(X)$. For instance, $X$ is discrete if and only if $ht(X)=1$.

 The following classical  theorem is due to A. Pe\l czy\'nski and Z. Semadeni.

\begin{theorem}{\rm (\cite[Theorem~8.5.4]{Semadeni})}\label{Theor:scattered}
A compact space $X$ is scattered if and only if there is no continuous mapping of $X$ onto the segment $[0,1]$.
\end{theorem}

A continuous surjection $\pi:X \rightarrow Y$ is called {\it irreducible} (see \cite[Definition~7.1.11]{Semadeni})
if for every closed subset $F$ of $X$ the condition $\pi(F)=Y$ implies $F=X$.

\begin{proposition}{\rm (\cite[Proposition~7.1.13]{Semadeni})}\label{Prop:restriction}
Let $X$ be a compact space and let $\pi:X \rightarrow Y$ be a continuous surjection.
Then there exists a closed subset $F$ of $X$ such that $\pi(F)=Y$ and the restriction ${\pi\res_{F}:F \rightarrow Y}$ is irreducible.
\end{proposition}

\begin{proposition}{\rm (\cite[Proposition~25.2.1]{Semadeni})}\label{Prop:dense}
Let $X$ be a compact space and let $\pi:X \rightarrow Y$ be a continuous surjection.
Then  $\pi$ is irreducible if and only if whenever  $E \subset X$ and $\pi(E)$ is dense in $Y$, then $E$ is dense in $X$.
\end{proposition}
 Recall that a Tychonoff space $X$ is \emph{\v{C}ech-complete} if $X$ is a $G_{\delta}$-set in some
(equivalently, any) compactification of $X$, (see \cite[3.9.1]{Engelking}).
It is well known that every locally compact space and every completely metrizable space is \v{C}ech-complete.
Next statement resolves an open question posed in \cite{FKLS}.

\begin{theorem}\label{Theor:Main_result}
Every \v{C}ech-complete (in particular, compact) $\dcal$-space is scattered.
\end{theorem}
\begin{proof}
\emph{Step 1: $X$ is compact}. On the contrary, assume that $X$ is not scattered.
First, by Theorem \ref{Theor:scattered}, there is a continuous mapping $\pi$ from $X$ onto the segment $[0,1]$.
Second, by Proposition \ref{Prop:restriction}, there exists a closed subset $F$ of $X$ such that $\pi(F)=[0,1]$ and the restriction ${\pi\res_{F}:F \rightarrow [0,1]}$ is irreducible.
Since $X \in \dcal$ the compact space $F$ also belongs to $\dcal$, by Corollary \ref{cor:subspace}. For simplicity, without loss of generality we may assume that
$F$ is $X$ itself and $\pi: X \rightarrow [0,1]$ is irreducible.

Let $\{X_{n}:n\in\omega\}$ be a partition of $[0,1]$ into dense sets. Put $Y_{n}=\bigcup_{k\geq n}X_{k}$, and $Z_{n}=\pi^{-1}(Y_{n})$ for all $n\in\omega$. Then all sets $Z_{n}$ are dense in $X$ by Proposition  \ref{Prop:dense} and the intersection $\bigcap_{n\in\omega} Z_{n}$ is empty. Every compact space $X$ is a Baire space, i.e. 
the Baire category theorem holds in $X$, hence if $\{U_{n}:n\in\omega\}$ is any open expansion of $\{Z_{n}:n\in\omega\}$, then the intersection $\bigcap_{n\in\omega}U_{n}$ is dense in $X$.
 In view of our Theorem \ref{Theor:description} this conclusion contradicts the assumption $X \in \dcal$, and the proof follows.

\emph{Step 2: $X$ is any \v{C}ech-complete space}. By the first step we deduce that every compact subset of $X$ is scattered.
But any \v{C}ech-complete space $X$ is scattered if and only if every compact subset of $X$ is scattered.
A detailed proof of this probably folklore statement can be found in \cite{STTWW}.
\end{proof}

\begin{proposition}\label{prop:count}
If $X$ is a first-countable compact space, then $X\in\dcal$ if and only if $X$ is countable.
\end{proposition}
\begin{proof}
If $X \in \dcal$, then $X$ is scattered, by Theorem \ref{Theor:Main_result}.
By the classical theorem
of S. Mazurkiewicz and W. Sierpi\'nski \cite[Theorem~8.6.10]{Semadeni}, a first-countable
compact space is scattered if and only if it is countable. This proves (i) $\Rightarrow$ (ii).
The converse is known  \cite{FKLS}
and follows from the fact that any countable space $X = \{x_n: n \in \omega\}$ admits a scant cover.
Indeed, define $X_n = \{x_i : i \geq n\}$. Then the family $\{X_n:n \in \omega\}$ is a scant cover of $X$.
 Now it suffices to mention Remark \ref{Rem:scant}.
\end{proof}

\begin{remark}\label{Knaster}
Theorem \ref{Theor:Main_result} extends also a well-known result of B. Knaster and K. Urbanik 
stating that every countable \v{C}ech-complete space is scattered \cite{KU}. It is easy to see that a 
countable Baire space contains a dense subset of isolated points, but in general does not have to be scattered.
We don't know whether every Baire $\dcal$-space must have isolated points.
\end{remark}

Recall that an Eberlein compact is a compact space homeomorphic to a subset of a Banach space with the weak topology.
A compact space is said to be a Corson compact space if it can be embedded in a $\Sigma$-product of the real lines.
Every Eberlein compact is Corson, but not vice versa.
However, every scattered Corson compact space is a scattered Eberlein compact space \cite{Alster}.
\begin{theorem}{\rm (\cite{FKLS})}\label{Theor:Corson}
A Corson compact space $X$ belongs to the class $\dcal$ if and only if $X$ is a scattered Eberlein compact space.
\end{theorem}
Bearing in mind Theorem \ref{Theor:Main_result}, to show Theorem  \ref{Theor:Corson} it suffices
 to  use the fact that every scattered Eberlein compact space admits a scant cover (the latter follows from the proof of \cite[Lemma 1.1]{BM})
and  then apply Remark \ref{Rem:scant}.


Being motivated by the previous results one can ask if there exist scattered compact spaces $X\in\Delta$ which are not Eberlein compact.
 The next question is also crucial: Does there exist a compact scattered space $X \notin \dcal$? 
Below we answer both questions positively.

We need the following somewhat technical

\begin{theorem} \label{Theor:union}
Let $Z = C_0 \cup C_1$ be a Tychonoff space such that
\begin{enumerate}
\item[{\rm (1)}] $C_0 \cap C_1 = \emptyset$.
\item[{\rm (2)}] $C_0$ is an open $F_{\sigma}$ subset of $Z$.
\item[{\rm (3)}] both $C_0$ and $C_1$ belong to the class $\dcal$.
\end{enumerate}
Then $Z$ also belongs to  the class $\dcal$.
\end{theorem}
\begin{proof} By assumption, $C_0 = \bigcup\{F_n: n \in \omega\}$, where each $F_n$ is closed in $Z$.
Let $\{X_n: n \in \omega\}$ be any countable collection of pairwise disjoint subsets of $Z$.
Our target is to define open sets $U_n \supseteq X_n$, $n \in \omega$ in such a way that 
the collection $\{U_n: n \in \omega\}$ is point-finite.
We decompose the sets $X_n = X_n^{0} \cup X_n^{1}$, where
$X_n^{0} = X_n \cap C_0$ and $X_n^{1} = X_n \cap C_1$.
By Theorem \ref{Theor:description},
the collection $\{X_n^{0}: n \in \omega\}$ expands to a point-finite open collection $\{U_n^{0}: n \in \omega\}$ in $C_0$.
The set $C_0$ is open in $Z$, therefore $U_n^{0}$ are open in $Z$ as well.

Now we consider the disjoint collection $\{X_n^{1}: n \in \omega\}$ in $C_1$. By assumption, $C_1 \in \dcal$, therefore applying Theorem \ref{Theor:description} once more,
we find a point-finite expansion $\{V_n^{1}: n \in \omega\}$ in $C_1$ consisting of sets which are open in $C_1$.
Every set $V_n^{1}$ is a trace of some set $W_n^{1}$, which is open in $Z$, i.e. $V_n^{1}= W_n^{1} \cap C_1$, and every $W_n^{1}$ is open in $Z$.
We refine the sets  $W_n^{1}$ by the formula  $U_n^{1} = W_n^{1} \setminus \bigcup\{F_i: i \leq n\}$.
Since all sets $F_i$ are closed in $Z$, the sets $U_n^{1}$ remain open in $Z$.
Since all sets $F_i$ are disjoint with $C_1$, the collection $\{U_n^{1}: n \in \omega\}$ remains to be an expansion of $\{X_n^{1}: n \in \omega\}$.
Furthermore, the collection $\{U_n^{1}: n \in \omega\}$ is point-finite, because $\{V_n^{1}: n \in \omega\}$ is point-finite,
and every point  $z \in C_0$ belongs to some $F_n$, hence $z \notin U_m^1$ for every $m \geq n$.
Finally, we define $U_n = U_n^{0} \cup U_n^{1}$.
The collection $\{U_n: n \in \omega\}$ is a point-finite open expansion of $\{X_n: n \in \omega\}$, and the proof is complete.
\end{proof}
This yields the following
\begin{corollary} \label{cor:height_2}
Let $Z$ be any separable scattered Tychonoff space such that its scattered height $ht(Z)$ is equal to 2.
Then $Z \in \dcal$.
\end{corollary}
\begin{proof} The structure of $Z$ is the following. $Z = C_0 \cup C_1$, where $C_0$ is a countable dense in $Z$ set consisting of isolated in $Z$ points and
$C_1$ consists of all accumulation points. Moreover, the space $C_1$ with the topology induced from $Z$ is discrete.
 All conditions of Theorem \ref{Theor:union} are satisfied, and the result follows.
\end{proof}

Our first example will be the one-point compactification of an Isbell--Mr\'owka space $\Psi(\acal)$.
We recall the construction and basic properties of $\Psi(\acal)$.
 Let $\acal$ be an almost disjoint family of subsets of the set of natural numbers $\N$ and let $\Psi(\acal)$ be the set $\N \cup \acal$
equipped with the topology defined as follows. For each $n \in \N$, the singleton $\{n\}$ is open, and
for each $A \in \acal$, a base of neighbourhoods of $A$ is the collection of all sets of the form
$\{A\} \cup B$, where $B\subset A$ and $|A \setminus B| < \omega$. The space $\Psi(\acal)$ is then a
 first-countable separable locally compact Tychonoff space.
If $\acal$ is a maximal almost disjoint (MAD) family, then the corresponding Isbell--Mr\'owka space $\Psi(\acal)$ would be in addition pseudocompact.
(Readers are advised to consult \cite[Chapter 8]{HT-MT} which surveys various topological properties of these spaces).

\begin{theorem} \label{Theor:height_3}
There exists a separable scattered compact space $X$ with the following properties:
\begin{enumerate}
\item[{\rm (a)}] The scattered height of $X$ is equal to 3.
\item[{\rm (b)}] $X \in \dcal$.
\item[{\rm (c)}] $X$ is not an Eberlein compact space.
\end{enumerate}
\end{theorem}
\begin{proof} Let $\acal$ be any uncountable almost disjoint (in particular, MAD) family of subsets of $\N$
and let $Z$ be the corresponding first-countable separable locally compact Isbell--Mr\'owka space $\Psi(\acal)$.
 It is easy to see that $Z  = \Psi(\acal)$ satisfies the assumptions of Corollary \ref{cor:height_2}.
Hence, $Z \in \dcal$. Now, denote by $X$ the one-point compactification of the separable locally compact space $Z$.
 Then the scattered height of $X$ is equal to 3. Note that  $X \in \dcal$ by Proposition \ref{prop:subspace}. Moreover,
$X$ is not an Eberlein compact space, since every separable Eberlein compact space is metrizable,
 while $\Psi(\acal)$ is metrizable if and only if $\acal$ is countable.
\end{proof}

Now we show that there exist scattered compact spaces which are not in the class $\dcal$.
We will use the classical Pressing Down Lemma.
Let $[0,\omega_1)$ be the set of all countable ordinals equipped with the order topology. For simplicity, we identify $[0,\omega_1)$ with $\omega_1$.
A subset $S$ of $\omega_1$ is called a {\it stationary} subset if $S$ has nonempty intersection with every closed and unbounded set in $\omega_1$.
A mapping $\varphi: S \rightarrow \omega_1$ is called {\it regressive} if $\varphi(\alpha) < \alpha$ for each $\alpha \in S$.
The proof of the following fundamental statement can be found for instance in \cite{Kunen}.

\begin{theorem}\textbf{Pressing Down Lemma.}\label{Theor:Press}
Let $\varphi: S \rightarrow \omega_1$ be a regressive mapping, where $S$ is a stationary subset of $\omega_1$.
Then for some $\gamma < \omega_1$, $\varphi^{-1}\lbrace{\gamma}\rbrace$ is a stationary subset of $\omega_1$.
\end{theorem}

It is known that there are plenty of stationary subsets of $\omega_1$.
In particular, every stationary set can be partitioned into countably many pairwise disjoint stationary sets \cite{Kunen}.
 Note that $\omega_1$ is a scattered locally compact and first-countable space. 
Next statement resolves an open question posed in \cite{FKLS}.
 \begin{theorem}\label{Theor:notD_1}
The compact scattered  space $[0,\omega_1]$ is not in the class $\dcal$.
\end{theorem}
\begin{proof} It suffices to show that $\omega_1$ does not belong to the class  $\dcal$.
Assume, on the contrary, that $\omega_1 \in \dcal$.
Denote by $L$ the set of all countable limit ordinals.
Evidently, $L$ is a closed unbounded set in $\omega_1$.
Take any representation of $L$ as the union of countably many pairwise disjoint stationary sets $\{S_n: n\in \omega\}$.
By Theorem \ref{Theor:description}, there exists a point-finite open expansion $\{U_n: n \in \omega\}$ in $\omega_1$.

Fr every  $\alpha \in U_n$ there is an ordinal $\beta(\alpha) < \alpha$ such that $[\beta(\alpha) , \alpha] \subset U_n$.
 In fact, for every $n \in \omega$ we can define a regressive mapping $\varphi_n: S_n \rightarrow \omega_1$
by the formula: $\varphi_n (\alpha) = \beta(\alpha)$. Since $S_n$ is a stationary set for every $n$, we can apply to $\varphi_n$ the Pressing Down Lemma.
Hence, for each $n$ there are a countable ordinal $\gamma_n$ and an uncountable subset $T_n \subset S_n$ with the following property:
$[\gamma_n , \alpha] \subset U_n$ for every $\alpha \in T_n$.
Denote $\gamma =\sup\{\gamma_n: n \in \omega\} \in \omega_1$. Because all $T_n$ are unbounded, for all natural $n$ we have an ordinal $\alpha_n \in T_n$
such that $\gamma < \alpha_n$ and $[\gamma_n , \alpha_n] \subset U_n$.
This implies that $\gamma \in U_n$ for every $n \in \omega$. However,  a collection $\{U_n: n \in \omega\}$ is point-finite.
The obtained contradiction finishes the proof.
\end{proof}

The function space $C_{k}(X)$ is called \emph{Asplund} if every separable vector subspace of $C_{k}(X)$ isomorphic to a Banach space,
 has the separable dual.
\begin{proposition}\label{Asplund}
If a Tychonoff space $X$ belongs to the class $\dcal$, then the space $C_{k}(X)$ is Asplund.
The converse conclusion fails in general.
\end{proposition}
\begin{proof}
Let $\mathcal{K}(X)$ be the family of all compact subset of $X$. By the assumption and Corollary \ref{cor:subspace},
 each  $K\in\mathcal{K}(X)$ belongs to the class $\dcal$. Clearly, $C_{k}(X)$ is  isomorphic to a (closed) subspace
of the  product $\Pi=\prod_{K\in\mathcal{K}(X)}C_{k}(K)$ of Banach spaces $C_{k}(K)$.
 Assume that $E$ is a separable vector subspace of $C_{k}(X)$ isomorphic to a Banach space.
 Observe that $E$ is isomorphic to a subspace of the finite product $\prod_{j\in F}C_{k}(K_{j})$ for  $K_{j}\in \mathcal{K}(X)$ and  $j\in F$.
Indeed, let $B$ be the unit (bounded) ball of the normed space $E$. Then there exists a finite set $F$ such that
 $\bigcap_{j\in F}\pi^{-1}_{j}(U_{j})\cap \Pi\subset B$, where $U_{j}$ are balls in spaces $C_{k}(K_{j})$, $j\in F$,
and $\pi_{j}$ are natural projections from $E$ onto $C_{k}(K_{j})$. Let $\pi_{F}$ be the (continuous) projection
 from $\Pi$ onto $\prod_{\j\in F}C_{k}(K_{j})$. Then  $\pi_{F}\restriction_E$ is an injective continuous and open map from $E$
onto $(\pi_{F}\restriction_E)(E)\subset\prod_{\j\in F}C_{k}(K_{j})$.  The injectivity of $\pi_{F}\restriction_E$  follows from the fact
 that $B$ is a bounded neighbourhood of zero in $E$. It is easy to see that the image $(\pi_{F}\restriction_E)(B)$
is an open neighbourhood of zero in $\prod_{j\in F}C_{k}(K_{j})$. On the other hand, $\prod_{j\in F}C_{k}(K_{j})$
is isomorphic to the space $C_{k}(\bigoplus_{j\in F}K_{j})$ and the compact space  $\bigoplus_{j\in F}K_{j}$ is  scattered.
 By the classical \cite[Theorem 12.29]{fabian} $E$ must have the separable dual $E^{*}$. Hence, $C_{k}(X)$ is Asplund.
The converse fails, as Theorem \ref{Theor:notD_1} shows for $X=[0,\omega_1]$.
\end{proof}
Since every infinite compact scattered space $X$ contains a nontrivial converging sequence, for such $X$ 
 the  Banach space $C(X)$ is not a Grothendieck space, (see \cite{dales}).
\begin{corollary}
If $X$ is an infinite compact and $X \in \dcal$, then the Banach space $C(X)$ is not a Grothendieck space.
 The converse fails, as $X=[0,\omega_{1}]$ applies.
\end{corollary}

For non-scattered spaces $X$ Theorem \ref{Theor:Main_result} implies immediately the following
\begin{corollary}\label{cor:beta}
If $X$ is a non-scattered space,  the Stone-\v{C}ech compactification  $\beta X$ is not in the class $\dcal$.
\end{corollary}

\begin{proposition}\label{prop:remainder}
Let $X=\beta Z \setminus Z$, where $Z$ is any infinite discrete space.
Then $X$ is not in the class $\dcal$.
\end{proposition}
\begin{proof}
$\beta Z \setminus Z$ does not have isolated points for any infinite discrete space $Z$.
\end{proof}

It is known that $X= [0, \omega_1]$ is the Stone-\v{C}ech compactification of $[0,\omega_1)$. We showed that $X \notin \dcal$.
Also, $\beta Z \notin \dcal$ for any infinite discrete space $Z$.
Every scattered Eberlein compact space belongs to the class $\dcal$ by Theorem \ref{Theor:Corson},  however
 no Eberlein compact $X$ can be the Stone-\v{C}ech compactification
$\beta Z$ for any proper subset $Z$ of $X$ by the Preiss--Simon theorem (see \cite[Theorem IV.5.8]{Arch}).
All these facts provides a motivation for the following result.

\begin{example}\label{example:beta}
There exists an Isbell--Mr\'owka space $Z$ which is {\it almost compact}
in the sense that the one-point compactification of $Z$ coincides with $\beta Z$ (see \cite[Theorem 8.6.1]{HT-MT}).
Define $X = \beta Z$. Then $X \in \dcal$, by Theorem \ref{Theor:height_3}.
\end{example}
\section{Metrizable spaces $X \in \dcal$}\label{metr}

 In this section we try to describe
constructively the structure of nontrivial metrizable spaces $X \in \dcal$.
Note first that every scattered metrizable $X$ is in the class $\dcal$ since every such space $X$
homeomorphically embeds into a scattered Eberlein compact \cite{BL}, and then Theorem \ref{Theor:Corson} and Corollary
\ref{cor:subspace} apply. We extend this result as follows.

A topological space $X$ is said to be \emph{$\sigma$-scattered}  if $X$ can be represented as a countable union of scattered subspaces
and $X$ is called \emph{strongly $\sigma$-discrete} if it is a union of countably many of its closed discrete subspaces.
Strongly $\sigma$-discreteness of $X$ implies that $X$ is $\sigma$-scattered, for any topological space.
For metrizable $X$, by the classical result of A. H. Stone \cite{Stone}, these two properties are equivalent.

\begin{proposition}\label{prop:sigma_scat} Any $\sigma$-scattered metrizable space belongs to the class $\Delta$.
\end{proposition}
\begin{proof} In view of aforementioned equivalence, every subset of $X$ is $F_{\sigma}$. 
If every subset of $X$ is $F_{\sigma}$, then $X \in \Delta$.
This fact apparently is well-known (see also a comment after Claim \ref{Claim1}).
For the sake of completeness we include a direct argument.
  We show that  $X$ satisfies  the condition (2) of Theorem \ref{Theor:description}.
 Let $\{X_n: n \in \omega \}$ be any countable disjoint partition of $X$.
Denote $X_n = \bigcup\{F_{n,m}: m \in \omega\}$, where each $F_{n,m}$ is closed in $X$.
Define open sets $U_n$ as follows: $U_0 = X$ and
$U_n = X \setminus \bigcup\{F_{k,m} : k < n, m < n\}$ for $n \geq 1$.
Then $\{U_n: n \in \omega \}$ is a point-finite open expansion of $\{X_n: n \in \omega \}$ in $X$.
\end{proof}

A metrizable space $A$ is called an \emph{absolutely analytic} if 
$A$ is homeomorphic to a Souslin subspace of a complete metric space $X$ (of an arbitrary weight), 
i.e. $A$ is expressible as 
$A = \bigcup_{\sigma \in \N^\N} \bigcap_{n \in \N} A_{\sigma|n}$,
where each $A_{\sigma|n}$ is a closed subset of $X$.
It is known that every absolutely analytic metrizable space $X$ (in particular, every Borel subspace of a complete metric space)
 either contains a homeomorphic copy of the Cantor set or it is strongly $\sigma$-discrete.
Therefore, for absolutely analytic metrizable space $X$ the converse is true: $X \in \dcal$ implies that $X$ is strongly $\sigma$-discrete \cite{FKLS}.

However, the last structural result can not be proved in general for all (separable) metrizable spaces without extra set-theoretic assumptions.
Let us recall several definitions of special subsets of the real line $\R$ (see \cite{Miller_survey}, \cite{Reed}).
\begin{enumerate}
\item[\rm (a)] A $Q$-set $X$ is a subset of $\R$ such that
each subset of $X$ is $F_{\sigma}$, or, equivalently, each subset of $X$ is $G_{\delta}$ in $X$.
\item[\rm (b)] A $\lambda$-set $X$  is a subset of $\R$ such that
each countable $A\subset X$ is $G_{\delta}$ in $X$.
\item[\rm (c)] A $\Delta$-set $X$ is a subset of $\R$ such that
 for every decreasing sequence $\{D_n: n \in \omega\}$
subsets of $X$ with empty intersection there is a decreasing expansion $\{V_n: n \in \omega\}$ consisting of open
subsets of $X$ with empty intersection.
\end{enumerate}

\begin{Claim}\label{Claim1}
The existence of an uncountable separable metrizable $\Delta$-space is equivalent to the existence of an uncountable $\Delta$-set.
\end{Claim}
\begin{proof} Note that every separable metrizable space homeomorphically embeds into a Polish space $\R^{\omega}$ 
and the latter space is a one-to-one continuous image of the set of irrationals $\P$. Therefore, if $M$ is an uncountable separable metrizable space, then there exist
 an uncountable set $X \subset \R$ and a one-to-one continuous mapping from $X$ onto $M$. It is easy to see that $X$ is a $\Delta$-set provided $M$ is a $\Delta$-space.
\end{proof}

Note that in the original definition of a $\Delta$-set, G. M. Reed used $G_{\delta}$-sets instead of open sets and
 E. van Douwen observed that these two versions are equivalent \cite{Reed}. From the original definition it is obvious that each $Q$-set must be a $\Delta$-set.
The fact that every $\Delta$-set is a $\lambda$-set is known as well.
K. Kuratowski showed that in ZFC there exist uncountable $\lambda$-sets.
The existence of an uncountable $Q$-set is one of the fundamental set-theoretical problems considered by many authors.
F. Hausdorff showed that the cardinality of an uncountable $Q$-set $X$ has to be strictly smaller than the continuum $\cont= 2^{\aleph_0}$,
 so in models of ZFC plus the Continuum Hypothesis (CH) there are no uncountable $Q$-sets.
Let us outline several known most relevant facts.

(1) Martin's Axiom plus the negation of the Continuum Hypothesis (MA $+ \lnot$CH) implies
 that every subset $X \subset \R$ of cardinality less than $\cont$ is a $Q$-set (see \cite{FM}).

(2) It is consistent that there is a $Q$-set $X$ such that its square $X^2$ is not a $Q$-set \cite{F_squares}.

(3) The existence of an uncountable $Q$-set is equivalent to the existence of an uncountable strong $Q$-set, i.e. a $Q$-set all finite
powers of which are $Q$-sets \cite{P}.

(4) No $\Delta$-set $X$ can have cardinality $\cont$ \cite{P1}.
Hence, under MA, every subset of $\R$ that is a $\Delta$-set is also a $Q$-set.
Recently we proved the following claim: If $X$ has a countable network and $|X|=\cont$,
then $C_p(X)$ is not distinguished \cite{FKLS}. In view of our Theorem \ref{Theor:description} this fact means that
 no $\Delta$-space $X$ with a countable network can have cardinality $\cont$.
\footnote{ The referee kindly informed the authors that
the last result can be derived easily from the actual argument of \cite{P1}.}

(5) It is consistent that there exists a $\Delta$-set $X$ that is not a $Q$-set \cite{Knight}.
Of course, there are plenty of nonmetrizable $\Delta$-spaces with non-$G_{\delta}$ subsets, in ZFC.

(6) An uncountable $\Delta$-set exists if and only if there exists a separable countably paracompact nonnormal Moore space (see \cite{FRW} and \cite{P1}).

Summarizing, the following conclusion is an immediate consequence of our Theorem \ref{Theor:description} and the known facts about $\Delta$-sets listed above.

\begin{corollary} \label{cor:Moore}\mbox{}
\begin{enumerate}
\item[\rm (1)] The existence of an uncountable separable metrizable space such that $C_p(X)$ is distinguished,
 is independent of  ZFC.
\item[\rm (2)] There exists an uncountable separable metrizable space $X$ such that $C_p(X)$ is distinguished,
if and only if there exists a separable countably paracompact nonnormal Moore space.
\end{enumerate}
\end{corollary}
\section{Basic operations in $\dcal$ and open problems}\label{problems}

In this section we consider the question whether the class $\Delta$ is invariant under the following basic topological operations:
 subspaces, continuous images, quotient continuous images, finite/countable unions, finite products.

\emph{1. Subspaces}. Trivial because of Corollary \ref{cor:subspace}.

\emph{2. (Quotient) continuous images}. Evidently, every topological space is a continuous image of a discrete one.
The following assertion is a consequence of a known fact about MAD families (see \cite[Chapter 8]{HT-MT}).
\begin{proposition}\label{prop:locomp}
There exists a first-countable separable pseudocompact locally compact Isbell--Mr\'owka space $\Psi(\acal)$
which admits a continuous surjection onto the segment $[0,1]$.
\end{proposition}
Thus, the class $\Delta$ is not invariant under continuous images even for separable locally compact spaces.
However, one can show that every uncountable quotient continuous image of any Isbell--Mr\'owka space $\Psi(\acal)$ satisfies the conditions of
Corollary \ref{cor:height_2}, therefore it is a $\Delta$-space.
Note also that a class of scattered Eberlein compact spaces preserves continuous images.
We were unable to resolve the following major open problem.

\begin{problem}\label{Problem_1} Let $X$ be any compact $\Delta$-space and $Y$ be a continuous image of $X$.
Is $Y$ a $\Delta$-space?
\end{problem}

Even a more general question is open.
\begin{problem}\label{Problem_2} Let $X$ be any $\Delta$-space and $Y$ be a quotient continuous image of $X$.
Is $Y$ a $\Delta$-space?
\end{problem}

Towards a solution of these problems we obtained several partial positive results.
\begin{proposition}\label{prop:quotientmap} Let $X$ be any $\Delta$-space and $\varphi: X \to Y$ be a quotient continuous surjection with only
finitely many nontrivial fibers. Then $Y$ is also a $\Delta$-space.
\end{proposition}
\begin{proof} By assumption, there exists a closed subset $K \subset X$ such that $\varphi(K)$ is finite and $\varphi\restriction_{X\setminus K}: X\setminus K \to Y\setminus \varphi(K)$ 
is a one-to-one mapping. Both sets $X\setminus K$ and $Y\setminus \varphi(K)$ are open in $X$ and $Y$, respectively. Since $\varphi$ is a quotient continuous mapping,
it is easy to see that $\varphi\restriction_{X\setminus K}$ is a homeomorphism. $X\setminus K$ is a $\Delta$-space, hence $Y\setminus \varphi(K)$ is also a $\Delta$-space.
 Finally, $Y$ is a $\Delta$-space, by Proposition \ref{prop:subspace}.
\end{proof}

\begin{proposition}\label{prop:closedmap} Let $X$ be any $\Delta$-space and $\varphi: X \to Y$ be a closed continuous surjection with finite fibers.
Then $Y$ is also a $\Delta$-space.
\end{proposition}
\begin{proof} Let $\{Y_n: n\in\omega\}$ be a partition of $Y$. By assumption, 
the partition $\{\varphi^{-1}(Y_n): n\in\omega\}$ admits a point-finite open expansion $\{U_n: n\in\omega\}$ in $X$.
 Clearly, $\varphi(X\setminus U_n)$ are closed sets in $Y$. Define $V_n = Y\setminus \varphi(X\setminus U_n)$ for each $n\in\omega$.
We have that $\{V_n: n\in\omega\}$ is an open expansion of $\{Y_n: n\in\omega\}$ in $Y$.
It remains to verify that the family $\{V_n: n\in\omega\}$ is point-finite.
Indeed, let $y\in Y$ be any point. Each point in the fiber $\varphi^{-1}(y)$ belongs to a finite number of  sets $U_n$. Since the fiber $\varphi^{-1}(y)$ is finite,
$y$ is contained only in a finite number of sets $V_n$ which finishes the proof.  
\end{proof}

\emph{3. Finite/countable unions}. 

\begin{proposition}\label{prop:unions} Assume that $X$ is a finite union of closed subsets $X_i$, where each $X_i$ belongs to the class $\Delta$.
Then $X$ also belongs to $\Delta$. In particular, a finite union of compact $\Delta$-spaces is also a $\Delta$-space.
\end{proposition}
\begin{proof} Denote by $Z$ the discrete finite union of $\Delta$-spaces $X_i$. Obviously, $Z$ is a $\Delta$-space which admits a natural
closed continuous mapping onto $X$. Since all fibers of this mapping are finite, the result follows from Proposition \ref{prop:closedmap}.
\end{proof}

We recall a definition of the Michael line. The Michael line $X$ is the refinement of the real line $\R$ obtained by isolating all irrational points.
So, $X$ can be represented as a countable disjoint union of singletons (rationals) and an open discrete set.
 Nevertheless, the Michael line $X$ is not in $\Delta$ \cite{FKLS}.
This example and Proposition \ref{prop:unions} justify the following

\begin{problem}\label{Problem_3} Let $X$ be a countable union of compact subspaces $X_i$ such that each $X_i$ belongs to the class $\Delta$.
Does $X$ belong to the class $\Delta$?
\end{problem}

\emph{4. Finite products}. We already mentioned earlier that the existence of a $Q$-set $X \subset \R$  such that its square $X^2$ is not a $Q$-set,
is consistent with ZFC. 

\begin{problem}\label{Problem_4} Is the existence of a $\Delta$-set $X \subset \R$ such that its square $X^2$ is not a $\Delta$-set,
consistent with ZFC?
\end{problem}

It is known that the finite product of scattered Eberlein compact spaces is a scattered Eberlein compact.
\begin{problem}\label{Problem_5} Let $X$ be the product of two compact spaces $X_1$ and $X_2$ such that each $X_i$ belongs to the class $\Delta$.
Does $X$ belong to the class $\Delta$?
\end{problem}

Our last problem is inspired by Theorem \ref{Theor:height_3}. 
\begin{problem}\label{Problem_6} Let $X$ be any scattered compact space with a finite scattered height.
Does $X$ belong to the class $\dcal$?
\end{problem}

\textbf{Acknowledgements.} The authors thank Michael Hru\v s\'ak for a useful information about Isbell--Mr\'owka spaces.



\begin{thebibliography}{99}

\bibitem{Alster} K. Alster,
\newblock \textit{Some remarks on Eberlein compacts}, Fund. Math. \textbf{104} (1979), 43--46.

\bibitem{Arch} A. V. Arkhangel'ski\u{\i},
\newblock \textit{Topological Function Spaces},
\newblock Kluwer, Dordrecht, 1992.

\bibitem{BL} T. Banakh and A. Leiderman,
\newblock \textit{Uniform Eberlein compactifications of metrizable spaces},
\newblock Topology and Appl. \textbf{159} (2012), 1691--1694.

\bibitem{BM} M. Bell and W. Marciszewski,
\newblock \textit{On scattered Eberlein compact spaces},
\newblock Israel J. Math. \textbf{158} (2007), 217--224.

\bibitem{BB1} D. Bierstedt and J. Bonet,
\newblock \textit{Density conditions in Fr\'{e}chet and }$(DF)$\textit{-spaces}, Rev. Mat. Univ. Complut. Madrid \textbf{2} (1989), 59--75.

\bibitem{BB2} D. Bierstedt and J. Bonet,
\newblock \textit{Some aspects of the modern theory of Fr\'{e}chet spaces}, RACSAM \textbf{97} (2003), 159--188.

\bibitem{dales} H. D. Dales, F. H. Dashiell Jr., A.T. -M. Lau, D. Strauss,
\newblock \textit{Banach Spaces of Continuous Functions as Dual Spaces}, CMS Books in Mathematics, Springer, 2016.

\bibitem{dieudonne} J. Dieudonn\'{e}, L. Schwartz, 
\textit{La dualit\'{e} dans les espaces $(\mathcal{F})$ et $(\mathcal{LF})$},
Ann. Inst. Fourier (Grenoble)  \textbf{1} (1949),  61--101.

\bibitem{Engelking} R. Engelking,
\newblock \textit{General Topology},
\newblock Heldermann Verlag, Berlin, 1989.

\bibitem{fabian} M. Fabian, P. Habala, P. Hajek, V. Montesinos Santalucia, J. Pelant, V. Zizler,
\newblock \textit{Functional Analysis and  Infinite-Dimensional Geometry}, CMS Books in Mathematics, Springer, 2001.

\bibitem{fe-ka} J. C. Ferrando, J. K\c akol,
\newblock \textit{Metrizable bounded sets in $C(X)$ spaces and distinguished $C_p(X)$ spaces}, J. Convex Anal. \textbf{26} (2019),  1337--1346.

\bibitem{fe-ka-sa} J. C. Ferrando, J.  K\c akol, S. A. Saxon, \textit{Examples of nondistinguished function dpaces $C_p(X)$},
 J. Convex. Anal. \textbf{26} (2019), 1347--1348.

\bibitem{FKLS} J. C. Ferrando, J. K\c akol, A. Leiderman, S. A. Saxon,
\newblock \textit{Distinguished $C_{p}\left( X\right)$ spaces}, RACSAM 115:27 (2021).

\bibitem{F_squares} W. G. Fleissner,
\newblock \textit{Squares of $Q$-sets},
\newblock Fund. Math. \textbf{118} (1983), 223--231.

\bibitem{FM} W. G. Fleissner and A. W. Miller,
\newblock \textit{On $Q$-sets}, Proc. Amer. Math. Soc. \textbf{78} (1980), 280--284.

\bibitem{FRW} W. G. Fleissner, G. M. Reed and M. Wage,
\newblock \textit{Normality versus countable paracompactness in perfect spaces}, Bull. Amer. Math. Soc. \textbf{82} (1976), 635--639.

\bibitem{grothendieck}  A. Grothendieck,
\newblock \textit{Sur les espaces $(F)$ et $(DF)$}, Summa Brasil. Math. \textbf{3} (1954), 57--123.

\bibitem{HT-MT} M. Hru\v s\'ak, \'A. Tamariz-Mascar\'ua, M. Tkachenko (Editors),
\newblock \textit{Pseudocompact Topological Spaces, A Survey of Classic and New Results with Open Problems},
\newblock Springer, 2018.


\bibitem{Ko} G. K\"{o}the,
\newblock \textit{Topological Vector Spaces, I}, Springer-Verlag, Berlin, 1983.

\bibitem{KU} B. Knaster, K. Urbanik, \textit{Sur les espaces complets s\'eparables de dimension $0$}, Fund. Math.
\textbf{40} (1953), 194--202.

\bibitem{Knight} R. W. Knight,
\newblock \textit{$\Delta$-Sets},
\newblock Trans. Amer. Math. Soc. \textbf{339} (1993), 45--60.

\bibitem{Kunen} K. Kunen,
\newblock \textit{Set Theory, An Introduction to Independence Proofs},
\newblock North-Holland, New York, 1980.

\bibitem{Miller_survey} A. W. Miller,
\newblock \textit{Special subsets of the real line}, in: Handbook of Set-Theoretic Topology,
\newblock North-Holland, Amsterdam, 1984, 201--233.

\bibitem{Nyikos} P. J. Nyikos,
\newblock \textit{A history of the normal Moore space problem},
\newblock in: Handbook of the History of General Topology, Volume 3, Kluwer, Dordrecht, 2001,  1179--1212.

\bibitem{P} T. C. Przymusi\'{n}ski,
\newblock \textit{The existence of $Q$-sets is equivalent to the existence of strong $Q$-sets},
\newblock Proc. Amer. Math. Soc. \textbf{79} (1980), 626--628.

\bibitem{P1} T. C. Przymusi\'{n}ski,
\newblock \textit{Normality and separability of Moore spaces},
\newblock in: Set-Theoretic Topology, Academic Press, New York, 1977, 325--337.

\bibitem{Reed} G. M. Reed,
\newblock \textit{On normality and countable paracompactness}, Fund. Math. \textbf{110} (1980), 145--152.

\bibitem{Semadeni} Z. Semadeni,
\newblock \textit{Banach Spaces of Continuous Functions, Volume I},
\newblock PWN - Polish Scientific Publishers, Warszawa, 1971.

\bibitem{STTWW} D. B. Shakhmatov, M. G. Tkachenko, V. V. Tkachuk, S. Watson, and R. G. Wilson,
\newblock \textit{Neither first countable nor \v{C}ech-complete spaces are maximal Tychonoff connected},
\newblock Proc. Amer. Math. Soc. \textbf{126} (1998), 279--287.

\bibitem{Stone} A. H. Stone,
\newblock \textit{Kernel constructions and Borel sets},
\newblock Trans. Amer. Math. Soc. \textbf{107} (1963), 58--70.

\end{thebibliography}
\end{document}